\documentclass[10pt, a4paper]{article}

\usepackage{amsmath, amsthm, amssymb, float}
\usepackage{graphicx,latexsym, amsmath,amsfonts}
\usepackage{amsthm}
\usepackage{hyperref}
\usepackage[nameinlink]{cleveref}
\usepackage{subcaption}
\usepackage{graphicx}
\usepackage{setspace}
\usepackage{makecell}
\usepackage{verbatim}
\usepackage{epstopdf}
\usepackage{color}
\usepackage{enumitem}
\usepackage{a4}
\usepackage[numbers]{natbib}
\allowdisplaybreaks[4]
\hypersetup{colorlinks=true}
\hypersetup{linkcolor=blue}

\numberwithin{equation}{section}
\theoremstyle{plain}

\newtheorem{Theorem}{Theorem}[section]
\newtheorem{Proposition}{Proposition}[section] 
 
\newtheorem{Lemma}{Lemma}[section]
\newtheorem{Problem}{Problem}[section]

\newtheorem{Assumption}{Assumption}[section]

\numberwithin{equation}{section}
\theoremstyle{plain}

\makeatletter \@addtoreset{equation}{section}

\begin{document}

\linespread{1.3}
\pagestyle{plain}

\title {Stochastic maximum principle for time-changed forward-backward stochastic control problem with Lévy noise}  

\author{Jingwei Chen\footnotemark[2] \footnotemark[1], Jun Ye\footnotemark[3], Feng Chen\footnotemark[4]}

\renewcommand{\thefootnote}{\fnsymbol{footnote}}
\footnotetext[1]{Corresponding author. (\texttt{chenj22@mails.tsinghua.edu.cn})}
\footnotetext[2]{Yau Mathematical Sciences Center, 
Tsinghua University,
Beijing, 100084, China.}
\footnotetext[3]{Department of Mathematical Sciences,
Tsinghua University,
Beijing, 100084, China.}
\footnotetext[4]{Department of Automation, Center for Brain-Inspired Computing Research, 
Tsinghua University,
Beijing, 100084, China.}

\date{}

\maketitle

\begin{abstract}
This paper establishes a stochastic maximum principle for optimal control problems governed by time-changed forward-backward stochastic differential equations with Lévy noise. The system incorporates a random, non-decreasing operational time (the inverse of an $\alpha$-stable subordinator) to model phenomena like trapping events and subdiffusion. Using a duality transformation and the convex variational method, we derive necessary and sufficient conditions for optimality, expressed through a novel set of adjoint equations. Finally, the theoretical results are applied to solve an explicit cash management problem under stochastic recursive utility.
\end{abstract}


\medskip
\noindent \textbf{AMS Subject Clasification:} 49K15,\;60H10

\section{Introduction}
Stochastic optimal control has emerged as a cornerstone tool in modern mathematical finance, engineering system optimization, and management science, underpinning the analysis and solution of complex decision-making problems under uncertainty. 
In classical control theory, the state dynamics of a controlled system are typically modeled via a forward stochastic differential equation (SDE), whose solution captures the evolution of the system's state over time in the presence of random perturbations. 
However, in a broad class of advanced applications, the state process itself evolves as an adapted stochastic process whose dynamics must be characterized by a backward stochastic differential equation (BSDE). 
This fundamental discrepancy between the forward evolution of the state and the backward evolution of the value process naturally gives rise to the theory of forward-backward stochastic differential equations (FBSDEs), which provides a unified mathematical framework to reconcile these dual dynamical structures (see \cite{wu1998FBSDE, peng1999FBSDE, shi2010SMP, wu2013FBSDE}).

On the other hand, the introduction of random time (or time change) into stochastic differential equations has become a vibrant and rapidly evolving research frontier (see \cite{2011Kobayashi,jingwei2025RINAM,jingwei2026}). 
By replacing the standard time increment $dt$ and Brownian motion increment $dB_t$ with the increments $dE_t$ and $dB_{E_t}$ of a random, non-decreasing operational time $dE_t$, this approach embeds stochastic dynamics within a flexible, time-transformed framework. 
In financial applications, this time-changed structure can effectively model phenomena such as the stagnation of asset prices during trading halts or the intermittent arrival of market information, where the “effective” time scale of the system deviates from calendar time. In physics, it provides a powerful tool to describe the trapping effects of particles in subdiffusive processes, where the movement of particles is constrained by irregular, time-inhomogeneous environments.

The study of stochastic control for time-changed stochastic control systems represents an even more recent and underdeveloped area. 
Nane and Ni \cite{nane2021time} made a pioneering contribution by establishing the stochastic maximum principle for a stochastic control problem driven by time-changed Lévy noise, extending classical control theory to accommodate jumps and time-inogeneohomus volatility. 
Jin and Song \cite{jin2024MFG} further advanced the field by deriving the stochastic maximum principle for a class of mean field game problems with time-changed Brownian motion.
Despite these advances, the interplay between FBSDEs and time-changed systems remains largely unexplored.

In this paper, we investigate the following system of time-changed forward-backward stochastic differential equations with Lévy noise (TCFBSDEwLN)

\begin{align}
  \begin{cases}
dX_{t}^{v}=f\left(t,E_{t},X_{t}^{v},v\left(t\right)\right)dE_{t}+\sigma\left(t,E_{t},X_{t}^{v},v\left(t\right)\right)dB_{E_{t}}\\
\quad\quad+\int_{|z|<c}b\left(t,E_{t},X_{t-}^{v},v\left(t\right),z\right)\tilde{N}\left(dz,dE_{t}\right),\\
-dY_{t}^{v}=\int_{|z|<c}g\left(t,E_{t},X_{t}^{v},Y_{t}^{v},A_{t}^{v},r^{v}\left(t,z\right),v\left(t\right)\right)\Pi\left(dz\right)dE_{t}-A_{t}^{v}dB_{E_{t}}\\
\quad\quad-\int_{|z|<c}r^{v}\left(t,z\right)\tilde{N}\left(dz,dE_{t}\right),\\
X_{0}^{v}=x_{0},\\
Y_{T}^{v}=\phi\left(X_{T}^{v}\right).
\end{cases}  \label{state equation}
\end{align}
Here, $E=\left(E_{t}\right)_{t\geq0}$ is the inverse of a subordinator $D=\left(D_{t}\right)_{t\geq0}$. 
The jumps of the subordinator $D_t$ represent random waiting periods or trapping events. 
Through the compensated Poisson random measure $\tilde{N}$, the system can capture discontinuous path changes caused by sudden events (e.g., market crashes, credit defaults, operational shocks).

The rest of the paper is organized as follows. In Section 2, we formulate the main problem and provide some preliminary facts on time-changed stochastic differential equations (TCSDEs). A duality method is employed to establish the existence of a dual state equation corresponding to the original state equation. In Section 3, the convex variational method is applied, and the main results of the paper are presented, including the derivation of the stochastic maximum principle. This principle furnishes both necessary and sufficient conditions for optimality. Finally, in Section 4, the theoretical results are illustrated through an application to a cash management problem under stochastic recursive utility.

\section{Problem formulation}
Let $(\Omega,\mathcal{F},\mathbb{P})$ be a complete probability space with filtration $\{\mathcal{F}_{t}\}_{t\geq0}$ satisfying the usual conditions (i.e. right continuous and increasing while $\mathcal{F}_{0}$ contains all $\mathbb{P}$-null sets).
Let $B=\left(B_{t}\right)_{t\geq0}$ be a standard Brownian motion.
Let $D$ be an $\alpha$-stable subordinator with stability index $0<\alpha<1$. A subordinator is a one-dimensional nondecreasing Lévy process with càdlàg paths starting at 0 with Laplace transform
\begin{equation}
  \mathbb{E}[e^{-\xi D_t}] = e^{-t \psi(\xi)}, \xi>0, t\geq 0, 
\end{equation}
where the Laplace exponent $\psi:(0,\infty)\rightarrow(0,\infty)$ is $\psi(\xi) = \int_0^\infty (1 - e^{-\xi y}) \Pi(dz), \xi >0$
and the Lévy measure $\Pi$ satisfies $\int_{0}^{\infty}(z\wedge1)\Pi(dz)<\infty$.
This paper focuses on the infinite Lévy measure case, i.e. $\Pi (0,\infty)=\infty$.  
Let $E=(E_{t})_{t\geq0}$ be the inverse of $D$, i.e.
\begin{equation}
    E_{t}:=\inf\{u>0:D_{u}>t\},t\geq0. 
\end{equation}
If $D$ is a stable subordinator, then $E$ has Mittag-Leffler distributions, see \cite{meerschaert2004limit}.

Let $\mathbb{E}_{B},\mathbb{E}_{D}$ and $\mathbb{E}$ denote the expectation under the probability measures $\mathbb{P}_{B}, \mathbb{P}_{D}$ and $\mathbb{P}$, respectively. 
Suppose $B$ and $E$ are mutually independent, then the product measure satisfies $\mathbb{P}=\mathbb{P}_{B}\times\mathbb{P}_{D}$.

In the content going forward, denote $C$ as generic positive constants that may change from line to line. 

For any given $s\in[0,T]$, we introduce the following spaces.
\begin{itemize}
  \item $L^{2}(\Omega,\mathcal{F}_{T}^{s};\mathbb{R}^{n})$: the space of $\mathcal{F}_{T}^{s}$-measurable $\mathbb{R}^{n}$-valued squared integrable random variables $\xi$  such that $\mathbb{E}\left[\left|\xi\right|^{2}\right]<\infty.$
  \item $L_{\mathcal{F}}^{2}([s,T];\mathbb{R}^{n})$: the space of $\mathcal{F}_{T}^{s}$-adapted $\mathbb{R}^{n}$-valued squared integrable processes $\varphi(t)$ such that $\mathbb{E}\int_{s}^{T}\left|\varphi(t)\right|^{2}dt<\infty.$
  \item $L_{\mathcal{F}}^{\infty}([s,T];\mathbb{R}^{n})$: the space of $\mathcal{F}_{T}^{s}$-adapted $\mathbb{R}^{n}$-valued essentially bounded processes such that $\left\Vert \varphi(\cdot)\right\Vert _{\infty}:=\underset{(t,\omega)\in[s,T]\times\Omega}{\mathrm{ess\ sup}}\left|\varphi_{t}(\omega)\right|<\infty.$
  \item $L_{\mathcal{F},p}^{2}([s,T];\mathbb{R}^{n})$: the space of $\mathcal{F}_{T}^{s}$-predictable $\mathbb{R}^{n}$-valued squared integrable processes such that $\mathbb{E}\int_{s}^{T}\left|\phi(t)\right|^{2}dt<\infty.$
  \item $F_{p}^{2}([s,T];\mathbb{R}^{n})$: the space of $\mathbb{R}^{n}-valued \mathcal{F}_{T}^{s}$-predictable processes $f(\cdot,\cdot,\cdot)$ defined on $\Omega\times[0,T]\times\boldsymbol{E}$ such that $\mathbb{E}\int_{0}^{T}\int_{\boldsymbol{E}}\left|f(\cdot,t,z)\right|^{2}\Pi(dz)dE_{t}<\infty.$
\end{itemize}

Let $\mathcal{U}$ be a nonempty convex subset of $\mathbb{R}^{k}$. We define the admissible control set
\begin{equation}
  \mathcal{U}_{ad}=\left\{ v(\cdot)\in L_{F,p}^{2}\left([0,T];\mathbb{R}^{k}\right);v(t)\in\mathcal{U},a.e.t\in[0,T],\mathbb{P}-a,s.\right\}. \notag 
\end{equation}

For any given admissible control 
$v(\cdot)\in\mathcal{U}_{ad}$ and intial condition $x_{0}\in\mathbb{R}^{n}$, we consider the time-changed forward-backward stochastic control system with Lévy noise as shown in \eqref{state equation}, or equivalently in the integral form:
\begin{align}
  \begin{cases}
X_{t}^{v}=x_{0}+\int_{0}^{t}f\left(s,E_{s},X_{s}^{v},v\left(s\right)\right)dE_{s}+\int_{0}^{t}\sigma\left(s,E_{s},X_{s}^{v},v\left(s\right)\right)dB_{E_{s}}\\
\quad\quad\quad+\int_{0}^{t}\int_{|z|<c}b\left(s,E_{s},X_{s-}^{v},v\left(s\right),z\right)\tilde{N}\left(dz,dE_{s}\right),\\
Y_{t}^{v}=\phi\left(X_{T}^{v}\right)+\int_{t}^{T}\int_{|z|<c}g\left(s,E_{s},X_{s}^{v},Y_{s}^{v},A_{s}^{v},r^{v}\left(s,z\right),v\left(s\right)\right)\Pi\left(dz\right)dE_{s}\\
\quad\quad\quad-\int_{t}^{T}A_{s}^{v}dB_{E_{s}}-\int_{t}^{T}\int_{|z|<c}r^{v}\left(s,z\right)\tilde{N}\left(dz,dE_{s}\right),
\end{cases} \label{state equation, integral form}
\end{align}

Define the cost functional as follows:
\begin{equation}
  J(v(\cdot)):=\mathbb{E}\left[\int_{0}^{T}\int_{|z|<c}l\left(t,E_{t},X_{t}^{v},Y_{t}^{v},A_{t}^{v},r^{v}\left(t,z\right),v\left(t\right)\right)\Pi\left(dz\right)dE_{t}+h(X_{T}^{v})+\gamma(Y_{0}^{v})\right]. \label{cost functional}
\end{equation}
where $l$ is the running cost, $\gamma$ is the initial cost, and  $h$ is the terminal cost. 

The optimal control problem is stated as follows.

\begin{Problem}
  Find an admissible control $u(\cdot)\in \mathcal{U}_{ad}$ satisfying
  \begin{equation}
    J(u(\cdot))=\inf_{v(\cdot)\in\mathcal{U}_{ad}}J(v(\cdot))
  \end{equation}
  subject to the state equation \eqref{state equation}.
\end{Problem}

The following assumptions are provided:
\begin{Assumption}\label{assumption}
  We assume the following assumptions hold:

  (i) $f,\sigma,b$ are global Lipschitz in $(x,v)$ and $g$ is global Lipschitz in $(x,y,a,r,v)$;

  (ii) $f,\sigma,b,g,l,h$ and $\gamma$ are continuous differentiable with respect to $(x,y,a,r,v)$;

  (iii) The derivatives of $f,\sigma,g and \int_{|z|<c}\left|b_{x}(\cdot,\cdot,\cdot,z)\right|^{2}\Pi\left(dz\right),\int_{|z|<c}\left|b_{v}(\cdot,\cdot,\cdot,z)\right|^{2}\Pi\left(dz\right)$ are bounded; 

  (iv) The derivatives of $l$ are bounded by $C\left(1+|x|+|y|+|a|+|r|+|v|\right)$; The derivatives of $h$ and $\gamma$  with respect to $x$ and $y$ are bounded by $C(1+|x|)$ and $C(1+|y|)$, respectively;

  (v) $\forall x\in\mathbb{R}^{n},\phi(x)\in L^{2}\left(\Omega,\mathcal{F}_{T};\mathbb{R}^{m}\right)$; and for fixed $\omega\in\Omega,\phi(x)$ is continuously differentiable in $x$,$\phi_{x}$ is bounded;

  (vi) For all $t_{1},t_{2}\in[0,T],f\left(t_{1},t_{2},0,0\right),g\left(t_{1},t_{2},0,0,0,0,0\right)\in L_{\mathcal{\mathcal{F}}}^{2}\left([0,T];\mathbb{R}^{n}\right),\sigma\left(t_{1},t_{2},0,0\right)\in L_{\mathcal{F},p}^{2}\left([0,T];\mathbb{R}^{n}\right)$ and $b\left(t_{1},t_{2},0,0,\cdot\right)\in\mathcal{F}_{p}^{2}\left([0,T];\mathbb{R}^{n}\right)$.
\end{Assumption}

\subsection{Preliminaries}
We also need the following lemmas to deal with time-changed systems.
The following lemma comes from Lemma 3.1 in \cite{nane2017stability}.
\begin{Lemma}(Itô formula for time-changed Lévy noise)\label{Ito's formula}
  Let $D_{t}$ be an RCLL subordinator and $E_{t}$ its inverse process.
  Define a filtration $\{\mathcal{G}_{t}\}_{t\geq0}$ by $\mathcal{G}_{t}=\mathcal{F}_{E_{t}}$. 
  Let $X$ be a process defined as follows:
  \begin{align}
    X_{t}=x_{0}+\int_{0}^{t}f\left(s,E_{s},X_{s-}\right)dt+\int_{0}^{t}k\left(s,E_{s},X_{s-}\right)dE_{t}&+\int_{0}^{t}g\left(s,E_{s},X_{s-}\right)dB_{E_{t}} \notag \\
	  \quad+\int_{0}^{t}\int_{|y|<c}h\left(s,E_{s},X_{s-},z\right)\tilde{N}(dE_{t},dz),
  \end{align}
  where $f,k,g,h$ are measurable functions such that all integrals are defined. 
  Here $c$ is the maximum allowable jump size.

  Then, for all $F:\mathbb{R}_{+}\times\mathbb{R}_{+}\times\mathbb{R}\rightarrow\mathbb{R}$ in $C^{1,1,2}(\mathbb{R}_{+}\times\mathbb{R}_{+}\times\mathbb{R},\mathbb{R})$, with probability one,
  \begin{align}
  	F(t,E_{t},X_{t})-F(0,0,x_{0})
	  &=\int_{0}^{t}L_{1}F\left(s,E_{s},X_{s-}\right)ds+\int_{0}^{t}L_{2}F\left(s,E_{s},X_{s-}\right)dE_{s} \notag \\
	  &\quad+\int_{0}^{t}\int_{|y|<c}\Big[F\left(s,E_{s},X_{s-}\right)+h\left(s,E_{s},X_{s-},z\right)-F\left(s,E_{s},X_{s-}\right)\Big]\tilde{N}(dE_{s},dz) \notag \\
	  &\quad+\int_{0}^{t}F_{x}\left(s,E_{s},X_{s-}\right)g\left(s,E_{s},X_{s-}\right)dB_{E_{s}}, 
  \end{align}
  where 
  \begin{equation}
    L_{1}F(t_{1},t_{2},x)	=F_{t_{1}}(t_{1},t_{2},x)+F_{x}(t_{1},t_{2},x)f(t_{1},t_{2},x), \notag
  \end{equation}
  and
  \begin{align}
    L_{2}F(t_{1},t_{2},x)	
    &=F_{t_{2}}(t_{1},t_{2},x)+F_{x}(t_{1},t_{2},x)k(t_{1},t_{2},x)+\frac{1}{2}g^{2}(t_{1},t_{2},x)F_{xx}(t_{1},t_{2},x) \notag \\
	  &\quad+\int_{|y|<c}\Big[F\left(t_{1},t_{2},x+h\left(t_{1},t_{2},x,z\right)\right)-F\left(t_{1},t_{2},x\right)-F_{x}\left(t_{1},t_{2},x\right)h(t_{1},t_{2},x,z)\Big]\Pi(dz). \notag
  \end{align}
\end{Lemma}

A process is said to be in synchronization with a time change $E$ if $Z$ is constant on every interval $\left[E_{t-},E_{t}\right]$ almost surely.
Denote $L\left(Z,\mathcal{F}_{t}\right)$ the class of -predictable process $U$ for which a stochastic integral driven by $Z$. 
The next lemma corresponds to Lemma 2.3 and Theorem 3.1 of \cite{2011Kobayashi}.
\begin{Lemma}(First and Second Change-of-variable Formula)\label{change-of-variable formula}
  Let $Z$ be an ($\mathcal{F}_{t}$)-semimartingale which is in synchronization with the continuous finite time change $E$.

  1. If $U\in L\left(Z,\mathcal{F}_{t}\right)$, then $U_{E_{t}}\in L(Z\circ E,\mathcal{F}_{E_{t}})$.
  Moreover, with probability one, for any $t\geq0$,
  \begin{equation}
    \int_{0}^{E_{t}}U_{s}dZ_{s}=\int_{0}^{t}U_{E_{s}}dZ_{E_{s}}; \quad  \notag
  \end{equation}

  2. If $U'\in L(Z\circ E,\mathcal{F}_{E_{t}})$, then $\left(U'_{D_{t-}}\right)\in L\left(Z,\mathcal{F}_{t}\right)$. 
   Moreover, with probability one, for any $t\geq0$,
  \begin{equation}
    \int_{0}^{t}U'_{s}dZ_{E_{s}}=\int_{0}^{E_{t}}U'_{D_{t-}}dZ_{s}. \notag
  \end{equation}
\end{Lemma}

The next lemma is a generalization of Grönwall's inequality, see Lemma 3.2 of \cite{li2023mckean}.
\begin{Lemma}\label{Gronwall's inequality}
  Suppose $D_{t}$ is a $\alpha$-stable subordinator and $E_{t}$ is the associated inverse stable subordinator. 
  Let $\ensuremath{T>0}$ and $u,m:\Omega\times\mathbb{R}_{+}\to\mathbb{R}_{+}$ be the $\mathcal{G}_{t}$-measurable functions which are integrable with respect to $E_{t}$. 
  Let $\ensuremath{n(t)}$ be a positive, monotonic, non-decreasing function. Then, the inequality 
  \begin{equation}
    u(t)\leq n(t)+\int_{0}^{t}m(s)u(s)\,dE_{s},\quad t\geq0
  \end{equation}
  implies almost surely 
  \begin{equation}
    u(t)\leq n(t)\exp\left\{ \int_{0}^{t}m(s)dE_{s}\right\} ,\quad t\geq0.
  \end{equation}
\end{Lemma}

\subsection{Duality}

We now introduce the following dual system to \eqref{state equation}:
\begin{align}
  \begin{cases}
dX_{t}^{v,*}=f\left(D_{t},t,X_{t}^{v,*},v\left(D_{t}\right)\right)dt+\sigma\left(D_{t},t,X_{t}^{v,*},v\left(D_{t}\right)\right)dB_{t}\\
\quad\quad\quad+\int_{|z|<c}b\left(D_{t},t,X_{t-}^{v,*},v\left(D_{t},t\right),z\right)\tilde{N}\left(dt,dz\right),\\
-Y_{t}^{v,*}=\int_{|z|<c}g\left(D_{t},t,X_{t}^{v,*},Y_{t}^{v,*},A_{D_{t}}^{v},r^{v}\left(D_{t},z\right),v\left(D_{t}\right)\right)\Pi\left(dz\right)dt\\
\quad\quad\quad-A_{D_{t}}^{v}dB_{t}-\int_{|z|<c}r^{v}\left(D_{t},z\right)\tilde{N}\left(dz,dt\right),\\
X_{0}^{v,*}=x_{0},\\
Y_{E_{T}}^{v,*}=\phi\left(X_{E_{T}}^{v,*}\right).
\end{cases} \label{dual state equation}
\end{align}

or equivalently in integral form:
\begin{align}
  \begin{cases}
X_{t}^{v,*}=x_{0}+\int_{0}^{E_{t}}f\left(D_{s},s,X_{s}^{v,*},v\left(D_{s}\right)\right)ds+\int_{0}^{E_{t}}\sigma\left(D_{s},s,X_{s}^{v,*},v\left(D_{s}\right)\right)dB_{s}\\
\quad\quad\quad+\int_{0}^{E_{t}}\int_{|z|<c}b\left(D_{s},s,X_{s-}^{v,*},v\left(D_{s},s\right),z\right)\tilde{N}\left(dz,ds\right),\\
Y_{t}^{v,*}=\phi\left(X_{E_{T}}^{*}\right)+\int_{t}^{E_{T}}\int_{|z|<c}g\left(D_{s},s,X_{s}^{v,*},Y_{s}^{v,*},A_{D_{s}}^{v},r^{v}\left(D_{s},z\right),v\left(D_{s}\right)\right)\Pi\left(dz\right)ds\\
\quad\quad\quad-\int_{t}^{E_{T}}A_{D_{s}}^{v}dB_{s}-\int_{t}^{E_{T}}\int_{|z|<c}r^{v}\left(D_{s},z\right)\tilde{N}\left(dz,ds\right),
\end{cases} \label{dual State equation-integral form}
\end{align}

\begin{Theorem}(Duality of the General State Equations)\label{duality of general state equations}
  Suppose there exists a strong solution $(X_{t}^{*},Y_{t}^{*},A(t),r(t,z))$ to the dual forward-backward stochastic control system with Lévy noise \eqref{dual State equation-integral form}, 
  then we have

  1. If $\left\{ X_{t}^{*},Y_{t}^{*},A_{D_{t}},r\left(D_{t},z\right)\right\} _{t\geq0}$ satisfies \eqref{dual State equation-integral form},
  then $\left\{ X_{E_{t}}^{*},Y_{E_{t}}^{*},A_{D_{E_{t}}},r\left(D_{E_{t}},z\right)\right\} _{t\geq0}$ satisfies \eqref{state equation, integral form}; 

  2. If $\left\{ X_{t},Y_{t},A_{t},r\left(t,z\right)\right\} _{t\geq0}$ satisfies \eqref{state equation, integral form}, then $\left\{ X_{D_{t}},Y_{D_{t}},A_{D_{t}},r\left(D_{t},z\right)\right\} _{t\geq0}$ satisfies \eqref{dual State equation-integral form}.
\end{Theorem}
\begin{proof}
  We first observe that $E_{s}$ is constant when $s\in\left[D_{E_{s-}},D_{E_{s}}\right]$. Hence, 
  \begin{equation}
    \int_{t}^{T}w\left(D_{E_{s}}\right)-w(s)dE_{s}=0 \quad \text{and} \quad \int_{t}^{T}w\left(D_{E_{s}}\right)-w(s)dB_{E_{s}}=0
  \end{equation}
  for any function or functional $w$.

  For the forward part, on one hand, by \autoref{change-of-variable formula}, we obtain
  \begin{align}
    X_{t}=X_{0}^{*}
    &=x_{0}-\int_{0}^{E_{t}}f\left(D_{s},s,X_{s}^{*},v\left(D_{s}\right)\right)ds-\int_{0}^{E_{t}}\sigma\left(D_{s},s,X_{s}^{*},v\left(D_{s}\right)\right)dB_{s} \notag \\
    &\quad -\int_{0}^{E_{t}}\int_{|z|<c}b\left(D_{s},s,X_{s}^{*},v\left(D_{s}\right),z\right)\tilde{N}\left(ds,dz\right) \notag \\
    &=x_{0}-\int_{0}^{t}f\left(D_{E_{s}},E_{s},X_{E_{s}}^{*},v\left(D_{E_{s}}\right)\right)dE_{s}-\int_{0}^{t}\sigma\left(D_{E_{s}},E_{s},X_{E_{s}}^{*},v\left(D_{E_{s}}\right)\right)dB_{E_{s}} \notag \\
    &\quad -\int_{0}^{t}\int_{|z|<c}b\left(D_{E_{s}},E_{s},X_{E_{s}}^{*},v\left(D_{E_{s}}\right),z\right)\tilde{N}\left(dE_{s},dz\right) \notag \\
    &=x_{0}-\int_{0}^{t}f\left(s,E_{s},X_{s},v\left(s\right)\right)dE_{s}-\int_{0}^{t}\sigma\left(s,E_{s},X_{s},v\left(s\right)\right)dB_{E_{s}} \notag \\
    &\quad -\int_{0}^{t}\int_{|z|<c}b\left(s,E_{s},X_{s},v\left(s\right),z\right)\tilde{N}\left(dE_{s},dz\right),
  \end{align}
  which solves \eqref{state equation, integral form}.

  On the other hand, 
  \begin{align}
    X_{t}&=x_{0}-\int_{0}^{t}f\left(s,E_{s},X_{s},v\left(s\right)\right)dE_{s}-\int_{0}^{t}\sigma\left(s,E_{s},X_{s},v\left(s\right)\right)dB_{E_{s}}\notag \\
    &\quad -\int_{0}^{t}\int_{|z|<c}b\left(s,E_{s},X_{s},v\left(s\right),z\right)\tilde{N}\left(dE_{s},dz\right) \notag \\
    &=x_{0}-\int_{0}^{E_{t}}f\left(D_{s-},E_{D_{s-}},X_{D_{s-}},v\left(D_{s-}\right)\right)ds-\int_{0}^{E_{t}}\sigma\left(D_{s-},E_{D_{s-}},X_{D_{s-}},v\left(D_{s-}\right)\right)dB_{s} \notag \\
    &\quad -\int_{0}^{E_{t}}\int_{|z|<c}b\left(D_{s-},E_{D_{s-}},X_{D_{s-}},v\left(D_{s-}\right),z\right)\tilde{N}\left(ds,dz\right) \notag \\
    &=x_{0}-\int_{0}^{E_{t}}f\left(D_{s},s,X_{D_{s}},v\left(D_{s}\right)\right)ds-\int_{0}^{E_{t}}\sigma\left(D_{s},s,X_{D_{s}},v\left(D_{s}\right)\right)dB_{s} \notag \\
    &\quad -\int_{0}^{E_{t}}\int_{|z|<c}b\left(D_{s},s,X_{D_{s}},v\left(D_{s}\right),z\right)\tilde{N}\left(ds,dz\right).
  \end{align}
  Observe that,
  \begin{align}
    X_{t}^{*}=X_{D_{t}}&x_{0}-\int_{0}^{E_{t}}f\left(D_{s},s,X_{D_{s}},v\left(D_{s}\right)\right)\Pi\left(dz\right)ds-\int_{0}^{E_{t}}\sigma\left(D_{s},s,X_{D_{s}},v\left(D_{s}\right)\right)dB_{s} \notag \\
    &\quad -\int_{0}^{E_{t}}\int_{|z|<c}b\left(D_{s},s,X_{D_{s}},v\left(D_{s}\right),z\right)\tilde{N}\left(ds,dz\right)
  \end{align}
  solves \eqref{state equation, integral form}.

  Similarly, for the backward part, applying \autoref{change-of-variable formula} leads to
  \begin{align}
    Y_{t}=Y_{T}^{*}
    &=\phi\left(X_{E_{T}}^{*}\right)-\int_{E_{t}}^{E_{T}}\int_{|z|<c}-g\left(D_{s},s,X_{s}^{*},Y_{s}^{*},A_{D_{s}},r\left(D_{s},z\right),v\left(D_{s}\right)\right)\Pi\left(dz\right)ds \notag \\
    &\quad -\int_{E_{t}}^{E_{T}}A_{D_{s}}dB_{s}-\int_{E_{t}}^{E_{T}}\int_{|z|<c}r\left(D_{s},z\right)\tilde{N}\left(ds,dz\right) \notag \\
    &=\phi\left(X_{E_{T}}\right)-\int_{t}^{T}\int_{|z|<c}-g\left(D_{E_{s}},E_{s},X_{E_{s}}^{*},Y_{E_{s}}^{*},A_{D_{E_{s}}},r\left(D_{E_{s}},z\right),v\left(D_{E_{s}}\right)\right)\Pi\left(dz\right)dE_{s} \notag \\
    &\quad -\int_{t}^{T}A_{D_{E_{s}}}dB_{E_{s}}-\int_{t}^{T}\int_{|z|<c}r\left(D_{E_{s}},z\right)\tilde{N}\left(dE_{s},dz\right) \notag \\
    &=\phi\left(X_{E_{T}}\right)-\int_{t}^{T}\int_{|z|<c}-g\left(s,E_{s},X_{s},Y_{s},A_{s},r\left(s,z\right),v\left(s\right)\right)\Pi\left(dz\right)dE_{s} \notag \\
    &\quad -\int_{t}^{T}A_{s}dB_{E_{s}}-\int_{t}^{T}\int_{|z|<c}r\left(s,z\right)\tilde{N}\left(dE_{s},dz\right). 
  \end{align}
  which solves \eqref{dual State equation-integral form}.

  On the other hand, 
  \begin{align}
    Y_{t}&=\phi\left(X_{T}\right)-\int_{t}^{T}\int_{|z|<c}-g\left(s,E_{s},X_{s},Y_{s},A_{s},r\left(s,z\right),v\left(s\right)\right)\Pi\left(dz\right)dE_{s}\notag \\
    &\quad -\int_{t}^{T}A_{s}dB_{E_{s}}-\int_{t}^{T}\int_{|z|<c}r\left(s,z\right)\tilde{N}\left(dE_{s},dz\right) \notag \\
    &=\phi\left(X_{T}\right)-\int_{E_{t}}^{E_{T}}\int_{|z|<c}-g\left(D_{s-},E_{D_{s-}},X_{D_{s-}},Y_{D_{s-}},A_{D_{s-}},r\left(D_{s-},z\right),v\left(D_{s-}\right)\right)\Pi\left(dz\right)ds \notag \\
    &\quad -\int_{E_{t}}^{E_{T}}A_{D_{s-}}dB_{s}-\int_{E_{t}}^{E_{T}}\int_{|z|<c}r\left(D_{s-},z\right)\tilde{N}\left(ds,dz\right) \notag \\
    &=\phi\left(X_{T}\right)-\int_{E_{t}}^{E_{T}}\int_{|z|<c}-g\left(D_{s},s,X_{D_{s}},Y_{D_{s}},A_{D_{s}},r\left(D_{s},z\right),v\left(D_{s}\right)\right)\Pi\left(dz\right)ds \notag \\
    &\quad -\int_{E_{t}}^{E_{T}}A_{D_{s}}dB_{s}-\int_{E_{t}}^{E_{T}}\int_{|z|<c}r\left(D_{s},z\right)\tilde{N}\left(ds,dz\right).
  \end{align}
  Observe that
  \begin{align}
    Y_{t}^{*}=Y_{D_{t}}&=\phi\left(X_{E_{T}}^{*}\right)-\int_{t}^{E_{T}}-g\left(D_{s},s,X_{D_{s}},Y_{D_{s}},A_{D_{s}},r\left(D_{s},z\right),v\left(D_{s}\right)\right)\Pi\left(dz\right)ds \notag \\
    &\quad -\int_{t}^{E_{T}}A_{D_{s}}dB_{s}-\int_{t}^{E_{T}}\int_{|z|<c}r\left(D_{s},z\right)\tilde{N}\left(ds,dz\right)
  \end{align}
  solves \eqref{dual State equation-integral form}.
\end{proof}

\begin{Theorem}\label{existence and uniqueness of dual system}
    Let \autoref{assumption} hold. 
    Then there exist a unique solution to the dual system \eqref{dual state equation}.
\end{Theorem}
\begin{proof}
    The key to the proof lies in adopting the "particle path" technique. By decomposing the problem into a product space and utilizing the independence between the processes $D$ and $B$, we transform a complex time-changed FBSDEwLN with coefficients depending on $D_t$ into a family of classical FBSDEwLNs whose coefficients are non-random and do not involve time changes. Then, using the theory of classical FBSDEs, we prove that the solution exists and is unique for each family of problems, thereby assembling the solution to the original problem.

    First, on the product space $\Omega=\Omega_{B}\times\Omega_{D}$, we define the original process $(X^{v,*},Y^{v,*},A^{v},r^{v})$.

    Next, for each fixed $\omega_{2}\in\Omega_{D}$, we define the marginal process on the space $\Omega=\Omega_{B}\times\Omega_{D}$ as follows:
    \begin{align}
      \begin{cases}
dX_{t}^{v,*,\omega_{2}}=f^{\omega_{2}}\left(D_{t}\left(\omega_{2}\right),t,X_{t}^{v,*,\omega_{2}},v\left(D_{t}\left(\omega_{2}\right)\right)\right)dt\\
\quad\quad\quad\quad+\sigma^{\omega_{2}}\left(D_{t}\left(\omega_{2}\right),t,X_{t}^{v,*,\omega_{2}},v\left(D_{t}\left(\omega_{2}\right)\right)\right)dB_{t}\\
\quad\quad\quad\quad+\int_{|z|<c}b^{\omega_{2}}\left(D_{t}\left(\omega_{2}\right),t,X_{t-}^{v,*,\omega_{2}},v\left(D_{t}\left(\omega_{2}\right)\right),z\right)\tilde{N}\left(dt,dz\right),\\
-Y_{t}^{\omega_{2}}=\int_{|z|<c}g^{\omega_{2}}\Big(D_{t}\left(\omega_{2}\right),t,X_{t}^{v,*,\omega_{2}},Y_{t}^{v,*,\omega_{2}},A_{D_{t}\left(\omega_{2}\right)}^{v,\omega_{2}},\\
\quad\quad\quad\quad r^{v,\omega_{2}}\left(D_{t}\left(\omega_{2}\right),z\right),v\left(D_{t}\left(\omega_{2}\right)\right)\Big)\Pi\left(dz\right)dt\\
\quad\quad\quad\quad-A_{D_{t}\left(\omega_{2}\right)}^{v,\omega_{2}}dB_{t}-\int_{|z|<c}r^{v,\omega_{2}}\left(D_{t}\left(\omega_{2}\right),z\right)\tilde{N}\left(dz,dt\right),\\
X_{0}^{v,*,\omega_{2}}=x_{0},\\
Y_{E_{T}}^{v,*,\omega_{2}}=\phi\left(X_{E_{T}}^{v,*,\omega_{2}}\right). 
\end{cases} \label{marginal equation}
    \end{align}
    The coefficients are obtained by replacing $D_{t}$ with $D_{t}(\omega_{2})$.

    Given the initial state $x_0$ and control $v(\cdot)$, under \autoref{assumption}, the forward part of the marginal system \eqref{marginal equation} admits a unique solution $X^v_{\cdot}$ (see \cite{Ikeda1989}).

After obtaining the unique solution $X^v_{\cdot}$ to the forward equation, we treat it as a known process and substitute it into the backward part of the marginal equation \eqref{marginal equation}. 
Since $X^v_{\cdot}$ is an adapted process, this BSDE admits a unique solution triplet $\big(Y^v_{\cdot}, A^v_{\cdot}, r^v_{\cdot}\big)$.

Since the above conclusion holds for almost all $\omega_{2}$, and the solution is measurable with respect to $\omega_{2}$, we define the process on the product space as:
\begin{align}
  &\quad \; \left(X^{v,*}(\omega_{1},\omega_{2}),Y^{v,*}(\omega_{1},\omega_{2}),A^{v}(\omega_{1},\omega_{2}),r^{v}(\omega_{1},\omega_{2})\right) \notag \\
  &=\left(X^{v,*,\omega_{2}}(\omega_{1}),Y^{v,*,\omega_{2}}(\omega_{1}),A^{v,\omega_{2}}(\omega_{1}),r^{v,\omega_{2}}(\omega_{1})\right) \notag
\end{align}
This quadruple $(X^{v,*},Y^{v,*},A^{v},r^{v})$ satisfies the original FBSDE system with $D_t$.
Uniqueness is also guaranteed by the uniqueness of the marginal solution \eqref{marginal equation}.
Therefore, under \autoref{assumption}, the dual system \eqref{dual state equation} admits a unique strong solution.
\end{proof}

Thus, by \autoref{existence and uniqueness of dual system} and \autoref{duality of general state equations}, there exists a unique solution to \eqref{state equation, integral form}, and the optimal control problem is well defined.

\section{Stochastic Maximum Principle}

\subsection{Convex variational method}

Next we utilize the classic convex variation method
(see \cite{bensoussan2006lectures,shi2010SMP}).
Let $u(\cdot)$ be an optimal control and let $\left(X_{\cdot},Y_{\cdot},A(\cdot),r\left(\cdot,\cdot\right)\right)$ be the corresponding optimal trajectory.
Let $v(\cdot)$ be such that $u(\cdot)+v(\cdot)\in\mathcal{U}$.
Since $\mathcal{U}$ is convex, then for any $0\leq\rho\leq1$, the perturbed control $u^{\rho}(\cdot):=u(\cdot)+\rho v(\cdot)$ is also in $\mathcal{U}$.
We define $\left(X_{t}^{1,\rho},Y_{t}^{1,\rho},A_{t}^{1,\rho},r^{1,\rho}\left(t,z\right)\right)$ as the trajectory corresponding to $u^{\rho}(\cdot)$:

\begin{align}
  &X_{t}^{\rho}=X_{t}+\rho X_{t}^{1,\rho}+\tilde{X}_{t}^{\rho}, 
  \quad\quad Y_{t}^{\rho}=Y_{t}+\rho Y_{t}^{1,\rho}+\tilde{Y}_{t}^{\rho}, \notag \\
  &A_{t}^{\rho}=A_{t}+\rho A_{t}^{1,\rho}+\tilde{A}_{t}^{\rho}, 
  \quad \quad \; \; r^{\rho}\left(t,z\right)=r\left(t,z\right)+\rho r^{1,\rho}\left(t,z\right)+\tilde{r}^{\rho}\left(t,z\right). \label{perturbed system}
\end{align}

where 
$\left(\tilde{X}_{t}^{\rho},\tilde{Y}_{t}^{\rho},\tilde{A}_{t}^{\rho},\tilde{r}^{\rho}\left(t,z\right)\right)$ is higher-order remainder terms of the perturbed solution;

and 
$\left(X_{t}^{1,\rho},Y_{t}^{1,\rho},A_{t}^{1,\rho},r^{1,\rho}\left(t,z\right)\right) $ is the solution of following variational equations:

\begin{align}
  \begin{cases}
dX_{t}^{1,\rho} & =\left[f_{x}\left(t,E_{t},X_{t},u\left(t\right)\right)X_{t}^{1,\rho}+f_{v}\left(t,E_{t},X_{t},u\left(t\right)\right)v\left(t\right)\right]dE_{t}\\
 & \quad+\left[\sigma_{x}\left(t,E_{t},X_{t},u\left(t\right)\right)X_{t}^{1,\rho}+\sigma_{v}\left(t,E_{t},X_{t},u\left(t\right)\right)v\left(t\right)\right]dB_{E_{t}}\\
 & \quad+\int_{|z|<c}\left[b_{x}\left(t,E_{t},X_{t-},u\left(t\right),z\right)X_{t}^{1,\rho}+b_{v}\left(t,E_{t},X_{t-},u\left(t\right),z\right)v\left(t\right)\right]\tilde{N}(dE_{t},dz),\\
-dY_{t}^{1,\rho} & =\int_{|z|<c}\Big[g_{x}\left(t,E_{t},z\right)X_{t}^{1,\rho}+g_{y}\left(t,E_{t},z\right)Y_{t}^{1,\rho}+g_{a}\left(t,E_{t},z\right)A_{t}^{1,\rho}+g_{r}\left(t,E_{t},z\right)R^{1,\rho}\left(t,z\right)\\
 & \quad\quad+g_{v}\left(t,E_{t},z\right)v\left(t\right)\Big]\Pi(dz)dE_{t}-A_{t}^{1,\rho}dB_{E_{t}}-\int_{|z|<c}r^{1,\rho}\left(t,z\right)\tilde{N}\left(dE_{t},dz\right),\\
X_{0}^{1,\rho} & =0,\\
Y_{T}^{1,\rho} & =\phi_{x}\left(X_{T}\right)X_{T}^{1,\rho}.
\end{cases}   \label{variational equaion}
\end{align}

Under \autoref{assumption}, there exists a unique $\left(X_{\cdot}^{1,\rho},Y_{\cdot}^{1,\rho},A_{\cdot}^{1,\rho},r^{1,\rho}(\cdot,\cdot)\right)\in L_{\mathcal{F}}^{2}([0,T];\mathbb{R}^{n})\times L_{\mathcal{F}}^{2}([0,T];\mathbb{R}^{m})\times L_{\mathcal{F},p}^{2}([0,T];\mathbb{R}^{m\times d})\times F_{p}^{2}([0,T];\mathbb{R}^{m})$ satisfying \eqref{variational equaion}.

From \eqref{perturbed system}, for $t\in[0,T]$, we have 
\begin{align}
  &\tilde{X}_{t}^{\rho}=\frac{X_{t}^{\rho}-X_{t}}{\rho}-X_{t}^{1,\rho}, 
  \quad \tilde{Y}_{t}^{\rho}:=\frac{Y_{t}^{\rho}-Y_{t}}{\rho}-Y_{t}^{1,\rho},  \notag \\
  &\tilde{A}_{t}^{\rho}:=\frac{A_{t}^{\rho}-A_{t}}{\rho}-A_{t}^{1,\rho}, 
  \quad \tilde{r}^{\rho}\left(t,z\right):=\frac{r^{\rho}\left(t,z\right)-r\left(t,z\right)}{\rho}-r^{1,\rho}\left(t,z\right). \label{high-order remainder terms}
\end{align}

Next, we obtain the convergence of the high-order remainder terms as the following:
\begin{Lemma}\label{Lemma 1}
  Let \autoref{assumption} hold. Then
  \begin{align}
    &\underset{\rho\rightarrow0}{lim}\underset{0\leq t\leq T}{sup}\mathbb{E}\left|\tilde{X}_{t}^{\rho}\right|^{2}=0,
    \quad \quad \underset{\rho\rightarrow0}{lim}\underset{0\leq t\leq T}{sup}\mathbb{E}\left|\tilde{Y}_{t}^{\rho}\right|^{2}=0, \notag \\
    &\underset{\rho\rightarrow0}{lim}\mathbb{E}\int_{0}^{T}\left|\tilde{A}_{t}^{\rho}\right|^{2}dE_{t}=0,
    \quad \underset{\rho\rightarrow0}{lim}\mathbb{E}\int_{0}^{T}\left|\tilde{r}^{\rho}\left(t,z\right)\right|^{2}\Pi(dz)dE_{t}=0.
  \end{align}
\end{Lemma}
\begin{proof}
  For the forward part, we have
  \begin{align}
    \begin{cases}
d\tilde{X}_{t}^{\rho}=\left[G^{1\rho}\left(t,E_{t})\right)\tilde{X}+G_{t}^{2\rho}\left(t,E_{t})\right)\right]dE_{t}+\left[G_{t}^{3\rho}\left(t,E_{t})\right)\tilde{X}_{t}^{\rho}+G_{t}^{4\rho}\left(t,E_{t})\right)\right]dB_{E_{t}}\\
\quad\quad+\int_{|z|<c}\left[G^{5\rho}\left(t-,E_{t},z\right)\tilde{X}_{t-}^{\rho}+G^{6\rho}\left(t-,E_{t},z\right)\right]\widetilde{N}\left(dz,dE_{t}\right),\\
\tilde{X}_{0}^{\rho}=0,
\end{cases}
  \end{align}
  where
  \begin{align}
    &G^{1\rho}\left(t,E_{t}\right):=\int_{0}^{1}f_{x}\left(t,E_{t},X_{t}+\lambda\rho\left(X_{t}^{1,\rho}+\tilde{X}_{t}^{\rho}\right),u(t)+\lambda\rho v(t)\right)d\lambda, \notag \\
    &G^{2\rho}\left(t,E_{t}\right):=\left[G^{1\rho}\left(t,E_{t}\right)-f_{x}\left(t,E_{t},X_{t},u(t)\right)\right]X_{t}^{1,\rho} \notag \\
    &\quad+\int_{0}^{1}\left[f_{v}\left(t,E_{t},X_{t},u(t)+\lambda\rho v(t)\right)-f_{v}\left(t,E_{t},X_{t},u(t)\right)\right]v(t)d\lambda, \notag \\
    &G^{3\rho}\left(t,E_{t}\right):=\int_{0}^{1}\sigma_{x}\left(t,E_{t},X_{t}+\lambda\rho\left(X_{t}^{1,\rho}+\tilde{X}_{t}^{\rho}\right),u(t)+\lambda\rho v(t)\right)d\lambda, \notag \\
    &G^{4\rho}\left(t,E_{t}\right):=\left[G^{3\rho}\left(t,E_{t}\right)-\sigma_{x}\left(t,E_{t},X_{t},u(t)\right)\right]X_{t}^{1,\rho} \notag \\
    &\quad+\int_{0}^{1}\left[\sigma_{v}\left(t,E_{t},X_{t},u(t)+\lambda\rho v(t)\right)-\sigma_{v}\left(t,E_{t},X_{t},u(t)\right)\right]v(t)d\lambda, \notag \\
    &G^{5\rho}\left(t-,E_{t},z\right):=\int_{0}^{1}b_{x}\left(t,E_{t},X_{t}+\lambda\rho\left(X_{t}^{1,\rho}+\tilde{X}_{t}^{\rho}\right),u(t)+\lambda\rho v(t)\right)d\lambda, \notag \\
    &G^{6\rho}\left(t-,E_{t},z\right):=\left[G^{5\rho}\left(t-,E_{t},z\right)-b_{x}\left(t,E_{t},X_{t},u(t)\right)\right]X_{t}^{1,\rho} \notag \\
    &\quad+\int_{0}^{1}\left[b_{v}\left(t,E_{t},X_{t},u(t)+\lambda\rho v(t)\right)-b_{v}\left(t,E_{t},X_{t},u(t)\right)\right]v(t)d\lambda. \notag
  \end{align}
  Applying Itô's formula to $\left|\tilde{X}_{t}^{\rho}\right|^{2}$ and using \autoref{assumption}, we obtain
  \begin{align}
    \mathbb{E}\left|\tilde{X}_{t}^{\rho}\right|^{2}&=\mathbb{E}\int_{0}^{T}\left[\langle2\tilde{X}_{t}^{\rho},G^{1\rho}\left(t,E_{t}\right)\tilde{X}_{t}^{\rho}+G^{2\rho}\left(t,E_{t}\right)\rangle+\left|G^{3\rho}\left(t,E_{t}\right)\tilde{X}_{t}^{\rho}+G^{4\rho}\left(t,E_{t}\right)\right|^{2}\right]dE_{t} \notag \\
    & \quad+\mathbb{E}\int_{0}^{T}\int_{|z|<c}|G^{5\rho}\left(t-,E_{t},z\right)\tilde{X}_{t}^{\rho}+G^{6\rho}\left(t-,E_{t},z\right)|^{2}\Pi(dz)dE_{t} \notag \\
    &\leq C\mathbb{E}\int_{0}^{T}|\tilde{X}_{t}^{\rho}|^{2}dE_{t}+o(\rho).
  \end{align}
  Applying \autoref{Gronwall's inequality}, we obtain the first result of the lemma.

  For the backward part, 
  \begin{align}
    \begin{cases}
-d\tilde{Y}_{t}^{\rho}=\int_{|z|<c}\Big[\Lambda^{1\rho}\left(t,E_{t},z\right)\tilde{X}_{t}^{\rho}+\Lambda^{2\rho}\left(t,E_{t},z\right)\tilde{Y}_{t}^{\rho}+\Lambda^{3\rho}\left(t,E_{t},z\right)\tilde{A}_{t}^{\rho}+\Lambda^{4\rho}\left(t,E_{t},z\right)\widetilde{r}^{\rho}(t,z)\\
\quad\quad\quad+\Lambda^{5\rho}\left(t,E_{t},z\right)\Big]\Pi(dz)dE_{t}-\tilde{A}_{t}^{\rho}dB_{E_{t}}-\int_{|z|<c}\widetilde{r}^{\rho}(t,z)\widetilde{N}\left(dz,dE_{t}\right),\\
\tilde{Y}_{T}^{\rho}=\rho^{-1}\left[\phi(X_{T}^{\rho})-\phi(X_{T})\right]-\phi_{x}(X_{T})X_{T}^{1,\rho},
\end{cases}
  \end{align}
  where
  \begin{align}
    \Lambda^{1\rho}\left(t,E_{t},z\right)	&:=\int_{0}^{1}g_{x}\Big(t,E_{t},X_{t}+\lambda\rho\left(X_{t}^{1,\rho}+\tilde{X}_{t}^{\rho}\right),Y_{t}+\lambda\rho\left(Y_{t}^{1,\rho}+\tilde{Y}_{t}^{\rho}\right), \notag \\
	  &\quad\quad A_{t}+\lambda\rho\left(A_{t}^{1,\rho}+\tilde{A}_{t}^{\rho}\right),r(t,z)+\lambda\rho\left(r^{1,\rho}(t,z)+\tilde{r}^{\rho}(r,z)\right),u(t)+\lambda\rho v(t),z\Big)d\lambda \notag \\
    \Lambda^{2\rho}\left(t,E_{t},z\right)&:=\int_{0}^{1}g_{y}\Big(t,E_{t},X_{t}+\lambda\rho\left(X_{t}^{1,\rho}+\tilde{X}_{t}^{\rho}\right),Y_{t}+\lambda\rho\left(Y_{t}^{1,\rho}+\tilde{Y}_{t}^{\rho}\right), \notag \\
    &\quad A_{t}+\lambda\rho\left(A_{t}^{1,\rho}+\tilde{A}_{t}^{\rho}\right),r(t,z)+\lambda\rho\left(r^{1,\rho}(t,z)+\tilde{r}^{\rho}(r,z)\right),u(t)+\lambda\rho v(t),z\Big)d\lambda \notag \\
    \Lambda^{3\rho}\left(t,E_{t},z\right)&:=\int_{0}^{1}g_{a}\Big(t,E_{t},X_{t}+\lambda\rho\left(X_{t}^{1,\rho}+\tilde{X}_{t}^{\rho}\right),Y_{t}+\lambda\rho\left(Y_{t}^{1,\rho}+\tilde{Y}_{t}^{\rho}\right), \notag \\
    &\quad A_{t}+\lambda\rho\left(A_{t}^{1,\rho}+\tilde{A}_{t}^{\rho}\right),r(t,z)+\lambda\rho\left(r^{1,\rho}(t,z)+\tilde{r}^{\rho}(r,z)\right),u(t)+\lambda\rho v(t),z\Big)d\lambda \notag \\
    \Lambda^{4\rho}\left(t,E_{t},z\right)&:=\int_{0}^{1}g_{r}\Big(t,E_{t},X_{t}+\lambda\rho\left(X_{t}^{1,\rho}+\tilde{X}_{t}^{\rho}\right),Y_{t}+\lambda\rho\left(Y_{t}^{1,\rho}+\tilde{Y}_{t}^{\rho}\right), \notag \\
    &\quad A_{t}+\lambda\rho\left(A_{t}^{1,\rho}+\tilde{A}_{t}^{\rho}\right),r(t,z)+\lambda\rho\left(r^{1,\rho}(t,z)+\tilde{r}^{\rho}(r,z)\right),u(t)+\lambda\rho v(t),z\Big)d\lambda \notag \\
    \Lambda^{5\rho}\left(t,E_{t},z\right)&:=\left[\Lambda^{1\rho}\left(t,E_{t},z\right)-g_{x}(\cdot)\right]X_{t}^{1,\rho}+\left[\Lambda^{2\rho}\left(t,E_{t},z\right)-g_{y}(\cdot)\right]Y_{t}^{1,\rho} \notag \\
    &\quad+\left[\Lambda^{3\rho}\left(t,E_{t},z\right)-g_{z}(\cdot)\right]A_{t}^{1,\rho}+\left[\Lambda^{4\rho}\left(t,E_{t},z\right)-g_{r}(\cdot)\right]r^{1,\rho}(t,z)\notag \\
    &\quad+\int_{0}^{1}\Big[g_{v}\Big(t,E_{t},X_{t}+\lambda\rho\left(X_{t}^{1,\rho}+\tilde{X}_{t}^{\rho}\right),Y_{t}+\lambda\rho\left(Y_{t}^{1,\rho}+\tilde{Y}_{t}^{\rho}\right), \notag \\
    &\quad \quad \quad A_{t}+\lambda\rho\left(A_{t}^{1,\rho}+\tilde{A}_{t}^{\rho}\right),r(t,z)+\lambda\rho\left(r^{1,\rho}(t,z)+\tilde{r}^{\rho}(r,z)\right),u(t)+\lambda\rho v(t)\Big) \notag \\
    &\quad \quad \quad -g_{v}\left(t,E_{t},X_{t},Y_{t},A_{t},r(t,z),u(t),z\right)\Big]v(t) d\lambda. \notag 
  \end{align}
  Applying Itô's formula to $\left|\tilde{Y}_{t}^{\rho}\right|^{2}$, noting \autoref{assumption}, we have
  \begin{align}
    &\quad \; \mathbb{E}\left|\tilde{Y}_{t}^{\rho}\right|^{2}+\mathbb{E}\int_{t}^{T}\left|\tilde{A}_{s}^{\rho}\right|^{2}dE_{s}+\mathbb{E}\int_{t}^{T}\int_{|z|<c}|\widetilde{r}^{\rho}(s,z)|^{2}\Pi(dz)dE_{s} \notag \\
    &=\mathbb{E}\int_{t}^{T}\int_{|z|<c}\Big\langle2\tilde{Y}_{s}^{\rho},\Lambda^{1\rho}\left(s,E_{s},z\right)\tilde{X}_{s}^{\rho}+\Lambda^{2\rho}\left(s,E_{s},z\right)\tilde{Y}_{s}^{\rho}+\Lambda^{3\rho}\left(s,E_{s},z\right)\tilde{A}_{s}^{\rho}+\Lambda^{4\rho}\left(s,E_{s},z\right)\widetilde{r}^{\rho}(s,z) \notag \\
    &\quad+\Lambda^{5\rho}\left(s,E_{s},z\right)\Big\rangle\Pi(dz)dE_{s} +\mathbb{E}\left[\rho^{-1}\left(\phi(X_{T}^{\rho})-\phi(X_{T})\right)-\phi_{x}(X_{T})X_{T}^{1,\rho}\right]^{2} \notag \\
    &\leq C\mathbb{E}\int_{t}^{T}\left|\tilde{Y}_{t}^{\rho}\right|^{2}dE_{s}+\frac{1}{2}\mathbb{E}\int_{t}^{T}\left|\tilde{A}_{s}^{\rho}\right|^{2}dE_{s}+\frac{1}{2}\mathbb{E}\int_{t}^{T}\int_{|z|<c}|\widetilde{r}^{\rho}(s,z)|^{2}\Pi(dz)dE_{s}+o(\rho).
  \end{align}
  Applying \autoref{Gronwall's inequality} again, we can get the last three convergence results.
\end{proof}

Since $u(\cdot)$ is an optimal control, then
\begin{equation}
  \rho^{-1}\left[J\left(u^{\rho}(\cdot)\right)-J\left(u(\cdot)\right)\right]\geq0. \label{pre-variational inequality}
\end{equation}

From this and \autoref{Lemma 1}, we have the following:

\begin{Lemma}(Variational inequality)\label{Lemma 2}
  Let \autoref{assumption} hold.
  Then the following variational inequality holds:
  \begin{align}
    o(\rho)	&\leq\mathbb{E}\int_{0}^{T}\int_{|z|<c}\Big[l_{x}\left(t,E_{t},X_{t},Y_{t},A_{t},r\left(t,z\right),u\left(t,E_{t}\right)\right)X_{t}^{1,\rho} \notag \\
	&\quad+l_{y}\left(t,E_{t},X_{t},Y_{t},A_{t},r\left(t,z\right),u\left(t\right)\right)Y_{t}^{1,\rho}  \notag \\
	&\quad+l_{a}\left(t,E_{t},X_{t},Y_{t},A_{t},r\left(t,z\right),u\left(t\right)\right)A_{t}^{1,\rho} \notag \\
	&\quad+l_{r}\left(t,E_{t},X_{t},Y_{t},A_{t},r\left(t,z\right),u\left(t\right)\right)r^{1,\rho}\left(t,z\right) \notag \\
	&\quad+l_{v}\left(t,E_{t},X_{t},Y_{t},A_{t},r\left(t,z\right),u\left(t\right)\right)v\left(t\right)\Pi\left(dz\right)\Big]dE_{t} \notag \\
	&\quad+\mathbb{E}\left[h_{x}(X_{T})X_{T}^{1,\rho}\right]+\mathbb{E}\left[\gamma_{y}(Y_{0})Y_{0}^{1,\rho}\right].
  \end{align}
\end{Lemma}

\begin{proof}
  For $h\Bigl(X_{T}^{\rho}\Bigr)-h\Bigl(X_{T}\Bigr)$, applying the first result of \autoref{Lemma 1}, we have
  \begin{align}
    \rho^{-1}\mathbb{E}\Bigl[h\Bigl(X_{T}^{\rho}\Bigr)-h\Bigl(X_{T}\Bigr)\Bigr] 
      & =\rho^{-1}\mathbb{E}\int_{0}^{1}h_{x}\Bigl(X_{T}+\lambda\Bigl(X_{T}^{\rho}-X_{T}\Bigr)\Bigr)\Bigl(X_{T}^{\rho}-X_{T}\Bigr)d\lambda \notag \\
    & \to\mathbb{E}\Bigl[h_{x}\Bigl(X_{T}\Bigr)X_{T}^{1,\rho}\Bigr]. \notag 
  \end{align}
  For $\gamma\Bigl(Y_{0}^{\rho}\Bigr)-\gamma\Bigl(Y_{0}\Bigr)$, similarly, we have
  \begin{align}
    \rho^{-1}\mathbb{E}\Bigl[\gamma\Bigl(Y_{0}^{\rho}\Bigr)-\gamma\Bigl(Y_{0}\Bigr)\Bigr] & =\rho^{-1}\mathbb{E}\int_{0}^{1}\gamma_{y}\Bigl(Y_{0}+\lambda\Bigl(Y_{0}^{\rho}-Y_{0}\Bigr)\Bigr)\Bigl(Y_{0}^{\rho}-Y_{0}\Bigr)d\lambda\notag \\
 & \to\mathbb{E}\Bigl[\gamma_{y}(Y_{0})Y_{0}^{1,\rho}\Bigr], \notag 
  \end{align}
  and 
  \begin{align} 
    &\quad \rho^{-1}\Bigl\{\mathbb{E}\int_{0}^{T}\int_{|z|<c}\Bigl[l\Bigl(t,E_{t},X_{t}^{\rho},Y_{t}^{\rho},A_{t}^{\rho},r^{\rho}(t,z),u(t)+\rho v(t)\Bigr)-l\left(t,E_{t},X_{t},Y_{t},A_{t},r(t,z),v(t)\right)\Bigr]\Pi(dz)dE_{t}\Bigr\} \notag \\
    &\to\mathbb{E}\int_{0}^{T}\int_{|z|<c}\Bigl[l_{x}\left(t,E_{t},X_{t},Y_{t},A_{t},r\left(t,z\right),u\left(t\right)\right)X_{t}^{1,\rho} \notag \\
 & \quad+l_{y}\left(t,E_{t},X_{t},Y_{t},A_{t},r\left(t,z\right),u\left(t\right)\right)Y_{t}^{1,\rho} \notag \\
 & \quad+l_{a}\left(t,E_{t},X_{t},Y_{t},A_{t},r\left(t,z\right),u\left(t\right)\right)A_{t}^{1,\rho} \notag \\
 & \quad+l_{r}\left(t,E_{t},X_{t},Y_{t},A_{t},r\left(t,z\right),u\left(t\right)\right)r^{1,\rho}\left(t,z\right) \notag \\
 & \quad+l_{v}\left(t,E_{t},X_{t},Y_{t},A_{t},r\left(t,z\right),u\left(t\right)\right)v\left(t\right)\Bigr]\Pi(dz)dE_{t}. \notag 
\\\end{align}
  Subsituite the above three results into \eqref{pre-variational inequality}, the proof is complete.
\end{proof}

Next, we introduce the following adjoint equations:

\begin{align}
  \begin{cases}
dp_{t}=\int_{|z|<c}\Big[g_{y}^{\mathsf{T}}\left(t,E_{t},X_{t},Y_{t},A_{t},r\left(t,z\right),u\left(t\right),z\right)p_{t}\\
\quad\quad-l_{y}^{\mathsf{T}}\left(t,E_{t},X_{t},Y_{t},A_{t},r\left(t,z\right),u\left(t\right)\right)\Big]\Pi\left(dz\right)dE_{t}\\
\quad\quad+\int_{|z|<c}\Big[g_{a}^{\mathsf{T}}\left(t,E_{t},X_{t},Y_{t},A_{t},r\left(t,z\right),u\left(t\right),z\right)p_{t}\\
\quad\quad-l_{a}^{\mathsf{T}}\left(t,E_{t},X_{t},Y_{t},A_{t},r\left(t,z\right),u\left(t\right)\right)\Big]\Pi\left(dz\right)dB_{E_{t}}\\
\quad\quad+\int_{|z|<c}\Big[g_{r}^{\mathsf{T}}\left(t,E_{t},X_{t-},Y_{t-},A_{t-},r\left(t,z\right),u\left(t\right),z\right)p_{t-}\\
\quad\quad-l_{r}^{\mathsf{T}}\left(t,E_{t},X_{t-},Y_{t-},A_{t-},r\left(t,z\right),u\left(t\right)\right)\Big]\tilde{N}\left(dE_{t},dz\right),\\
-dq_{t}=\Big[f_{x}^{\mathsf{T}}\left(t,E_{t},X_{t},Y_{t},A_{t},r\left(t,z\right),u\left(t\right)\right)q_{t}\\
\quad\quad-\int_{|z|<c}g_{x}^{\mathsf{T}}\left(t,E_{t},X_{t},Y_{t},A_{t},r\left(t,z\right),u\left(t\right),z\right)p_{t}\Pi\left(dz\right)\\
\quad\quad+\sigma_{x}^{\mathsf{T}}\left(t,E_{t},X_{t},Y_{t},A_{t},r\left(t,z\right),u\left(t\right)\right)k_{t}\\
\quad\quad+\int_{|z|<c}b_{x}^{\mathsf{T}}\left(t,E_{t},X_{t},Y_{t},A_{t},r\left(t,z\right),u\left(t\right),z\right)R\left(t,z\right)\\
\quad\quad+l_{x}^{\mathsf{T}}\left(t,E_{t},X_{t},Y_{t},A_{t},r\left(t,z\right),u\left(t\right)\right)\Pi\left(dz\right)\Big]dE_{t}\\
\quad\quad-k_{t}dB_{E_{t}}-\int_{|z|<c}R\left(t,z\right)\tilde{N}\left(dE_{t},dz\right),\\
p_{0}=-\gamma_{y}(y_{0}),\\
q_{T}=-\phi_{x}^{\mathsf{T}}\left(X_{T}\right)p_{T}+h_{x}\left(X_{T}\right).
\end{cases}   \label{adjoint state equation}
\end{align}

Similarly, under \autoref{assumption}, there exists a unique solution $\left(p_{\cdot},q_{\cdot},k_{\cdot},R\left(\cdot,\cdot\right)\right)\in L_{\mathcal{F}}^{2}\left([0,T];\mathbb{R}^{m}\right)\times L_{\mathcal{F}}^{2}\left([0,T];\mathbb{R}^{n}\right)\times L_{\mathcal{F,p}}^{2}\left([0,T];\mathbb{R}^{n\times d}\right)\times F_{p}^{2}\left([0,T];\mathbb{R}^{n}\right)$ satisfying \eqref{adjoint state equation}.

We define the Hamiltonian function $H:[0,T]\times\mathbb{R}_{+}\times\mathbb{R}^{n}\times\mathbb{R}^{m}\times\mathbb{R}^{m\times d}\times\mathbb{R}^{m}\times\mathcal{U}\times\mathbb{R}^{m}\times\mathbb{R}^{n}\times\mathbb{R}^{n\times d}\times\mathbb{R}^{n}\rightarrow\mathbb{R} $
as follows:
\begin{align}
  &\quad \; H\left(t_{1},t_{2},x,y,a,r(\cdot),v,p,q,k,R(\cdot)\right) \notag \\
  &:=\left\langle q,f\left(t_{1},t_{2},x,v\right)\right\rangle +\left\langle k,\sigma\left(t_{1},t_{2},x,v\right)\right\rangle   \notag \\
	&\quad-\int_{|z|<c}\Big[\left\langle p,g\left(t_{1},t_{2},x,y,a,r(z),v\right)\right\rangle -l\left(t_{1},t_{2},x,y,a,r(z),v\right) \notag \\
	&\quad-\left\langle R(z),b\left(t_{1},t_{2},x,v,z\right)\right\rangle \Big]\Pi(dz). \label{Hamiltonian}
\end{align}
or 
\begin{align}
  &\quad \; H\left(t,E_{t},X_{t},Y_{t},A_{t},r(t,z),v(t),p_{t},q_{t},k_{t},R(t,z)\right) \notag \\
  &:=\left\langle q_{t},f\left(t,E_{t},X_{t},v(t)\right)\right\rangle +\left\langle k_{t},\sigma\left(t,E_{t},X_{t},v(t)\right)\right\rangle    \notag \\
	&\quad-\int_{|z|<c}\Big[\left\langle p_{t},g\left(t,E_{t},X_{t},Y_{t},A_{t},r(t,z),v(t)\right)\right\rangle -l\left(t,E_{t},X_{t},Y_{t},A_{t},r(t,z),v(t)\right) \notag \\
	&\quad-\left\langle R(t,z),b\left(t,E_{t},X_{t},v(t),z\right)\right\rangle \Big]\Pi(dz). \label{Hamiltonian2}
\end{align}

Then, the adjoint equations \eqref{adjoint state equation} can be rewritten in the Hamiltonian's form:
\begin{align}
  \begin{cases}
dp_{t}=H_{y}\left(t,E_{t},X_{t},Y_{t},A_{t},r\left(t,\cdot\right),u\left(t\right),p_{t},q_{t},k_{t},R\left(t,\cdot\right)\right)dE_{t}\\
\quad\quad\quad+H_{z}\left(t,E_{t},X_{t},Y_{t},A_{t},r\left(t,\cdot\right),u\left(t\right),p_{t},q_{t},k_{t},R\left(t,\cdot\right)\right)dB_{E_{t}}\\
\quad\quad\quad+\int_{|z|<c}H_{r}\left(t,E_{t},X_{t},Y_{t},A_{t},r\left(t,\cdot\right),u\left(t\right),p_{t},q_{t},k_{t},R\left(t,\cdot\right)\right)\tilde{N}\left(dE_{t},dz\right),\\
-dq_{t}=H_{x}\left(t,E_{t},X_{t},Y_{t},A_{t},r\left(t,\cdot\right),u\left(t\right),p_{t},q_{t},k_{t},R\left(t,\cdot\right)\right)dE_{t}\\
\quad-k_{t}dB_{E_{t}}-\int_{|z|<c}R\left(t,z\right)\tilde{N}\left(dE_{t},dz\right),\\
p_{0}=-\gamma_{y}(y_{0}),\\
q_{T}=-\phi_{x}^{\mathrm{T}}\left(X_{T}\right)p_{T}+h_{x}\left(X_{T}\right).
\end{cases}  \label{adjoint state equation-Hamiltonian's form}
\end{align}

\subsection{Main results}
\begin{Theorem}(Necessary condition)\label{thm: Necessary condition}
  Let \autoref{assumption} hold.
  Let $u(\cdot)$ be an optimal control and $\left(X_{\cdot},Y_{\cdot},A_{\cdot},r(\cdot,\cdot)\right)$ be the corresponding trajectory. Then we have
  \begin{equation}
    \left(H_{v}\left(t,E_{t},X_{t},Y_{t},A_{t},r\left(t,\cdot\right),u\left(t\right),p_{t},q_{t},k_{t},R\left(t,\cdot\right)\right),v-u\left(t\right)\right)\geq0,\forall v\in\mathcal{U},a.e.,\mathbb{P}-a.s., \label{maximum condition}
  \end{equation}
  where Hamiltonian H is defined in \eqref{Hamiltonian}.
\end{Theorem}
\begin{proof}
  Applying Itô's formula (\autoref{Ito's formula}) to $\langle X_{t}^{1,\rho},q_{t}\rangle+\langle Y_{t}^{1,\rho},p_{t}\rangle $, we obtain
  \begin{align}
    &\quad \;\mathbb{E}\Big[h_{x}(X_{T})X_{T}^{1,\rho}\Big]+\mathbb{E}\Big[\gamma_{y}(Y_{0})Y_{0}^{1,\rho}\Big]	 \notag \\
    &=\mathbb{E}\int_{0}^{T}\int_{|z|<c}\Big[-l_{x}\left(t,E_{t},X_{t},Y_{t},A_{t},u(t),r(t,\cdot)\right)X_{t}^{1,\rho}-l_{y}\left(t,E_{t},X_{t},Y_{t},A_{t},u(t),r(t,\cdot)\right)Y_{t}^{1,\rho} \notag \\
	  &\quad-l_{a}\left(t,E_{t},X_{t},Y_{t},A_{t},u(t),r(t,\cdot)\right)A_{t}^{1,\rho}-l_{r}\left(t,E_{t},X_{t},Y_{t},A_{t},u(t),r(t,\cdot)\right)r^{1,\rho}(t,z) \notag \\
	  &\quad-l_{v}\left(t,E_{t},X_{t},Y_{t},A_{t},u(t),r(t,\cdot)\right)v(t)\Big]\Pi(dz)dE_{t} \notag \\
	  &\quad+\mathbb{E}\int_{0}^{T}\langle H_{v}\left(t,E_{t},X_{t},Y_{t},A_{t},u\left(t\right),p_{t},k_{t},q_{t},r(\cdot,t)\right),v\left(t\right)\rangle dE_{t}. \notag 
  \end{align}
  Then applying the variational inequality (\autoref{Lemma 2}), for $v(\cdot)$ such that $u(\cdot)+v(\cdot)\in\mathcal{U}_{ad}$,
  \begin{equation}
    \mathbb{E}\int_{0}^{T}\langle H_{v}\left(t,E_{t},X_{t},Y_{t},A_{t},r(\cdot,t),u\left(t\right),p_{t},k_{t},q_{t},R(\cdot,t)\right),v\left(t\right)\rangle dE_{t}\geq0.  \notag 
  \end{equation}
\end{proof}

\begin{Assumption}\label{assumption2}
  $h$ is convex in $x$, $\gamma$ is convex in $y$.
\end{Assumption}

\begin{Theorem}(Sufficient condition)
  Let \autoref{assumption} and \autoref{assumption2} hold.
  Let $u(\cdot)$ be an admissible control and $\left(X_{\cdot},Y_{\cdot},A_{\cdot},r(\cdot,\cdot)\right)$ be the corresponding trajectory with $Y_{T}=M_{T}X_{T},M_{T}\in\mathbb{R}^{m\times n}$.
  Let $(p_{\cdot},q_{\cdot},k_{\cdot},R(\cdot,\cdot))$ be the solution of the adjoint equations \eqref{adjoint state equation-Hamiltonian's form}.

  Suppose that $H$ is convex in $\left(x,y,a,r(\cdot),v\right)$.
  Then $u(\cdot)$ is an optimal control if it satisfies \eqref{maximum condition}.
\end{Theorem}
\begin{proof}
  Let $v(\cdot)$ be an arbitrary admissible control and $\left(X_{t}^{v},Y_{t}^{v},A_{t}^{v},r\left(t,z\right)\right)$ be the corresponding trajectory. 
  We consider
  \begin{align}
    J(u(\cdot))-J(v(\cdot))
    &=\mathbb{E}\int_{0}^{T}\int_{|z|<c}\left[l\left(t,E_{t},X_{t},Y_{t},A_{t},r\left(t,z\right),u(t)\right)-l\left(t,E_{t},X_{t}^{v},Y_{t}^{v},A_{t}^{v},r^{v}\left(t,z\right),v(t)\right)\right]\Pi\left(dz\right)dE_{t} \notag \\
	  &+\mathbb{E}\left[h(X_{T})-h(X_{T}^{v})\right]+\mathbb{E}\left[\gamma(Y_{0})-\gamma(Y_{0}^{v})\right] \notag \\
	  &=I_{1}+I_{2}, \notag
  \end{align} 
  where
  \begin{equation}
    I_{1}=\mathbb{E}\int_{0}^{T}\int_{|z|<c}\left[l\left(t,E_{t},X_{t},Y_{t},A_{t},r\left(t,z\right),u(t)\right)-l\left(t,E_{t},X_{t}^{v},Y_{t}^{v},A_{t}^{v},r^{v}\left(t,z\right),v(t)\right)\right]\Pi\left(dz\right)dE_{t} \notag
  \end{equation}
  and 
  \begin{equation}
    I_{2}=\mathbb{E}\left[h(X_{T})-h(X_{T}^{v})\right]+\mathbb{E}\left[\gamma(Y_{0})-\gamma(Y_{0}^{v})\right]. \notag
  \end{equation}
  For $I_1$, by definition of Hamiltonian, we have
  \begin{align}
    I_{1}
    &=\mathbb{E}\int_{0}^{T}\int_{|z|<c}\left[l\left(t,E_{t},X_{t},Y_{t},A_{t},r\left(t,z\right),u(t)\right)-l\left(t,E_{t},X_{t}^{v},Y_{t}^{v},A_{t}^{v},r\left(t,z\right),v(t)\right)\right]\Pi\left(dz\right)dE_{t}  \notag \\
	  &=\mathbb{E}\int_{0}^{T}\Big[H\left(t,E_{t},X_{t},Y_{t},A_{t},r(t,\cdot),u(t),p_{t},q_{t},k_{t},R(t,\cdot)\right) \notag \\
    &\quad-H\left(t,E_{t},X_{t}^{v},Y_{t}^{v},A_{t}^{v},r^{v}(t,\cdot),v(t),p_{t},q_{t},k_{t},R(t,\cdot)\right)\Big]dE_{t} \notag \\
    &\quad+\mathbb{E}\int_{0}^{T}\int_{|z|<c}\Big[-\langle q_{t},f\left(t,E_{t},X_{t},u(t)\right)-f\left(t,E_{t},X_{t}^{v},v(t)\right)\rangle \notag \\
    &\quad-\langle k_{t},\sigma\left(t,E_{t},X_{t},u(t)\right)-\sigma\left(t,E_{t},X_{t}^{v},v(t)\right)\rangle \notag \\
    &\quad-\langle R\left(t,z\right),b\left(t,E_{t},X_{t},u(t),z\right)-b\left(t,E_{t},X_{t}^{v},v(t),z\right)\rangle\Big]\Pi(dz)dE_{t} \notag 
  \end{align}
For $I_2$, by \autoref{assumption2} and Itô's formula, we have
\begin{align}
  &\quad \; \mathbb{E}\left[h\left(X_{T}\right)-h\left(X_{T}^{v}\right)\right]	 \notag \\ 
  &\leq\mathbb{E}\left[\left(X_{T}-X_{T}^{v}\right)^{\top}h_{x}\left(X_{T}\right)\right] \notag \\
	&=\mathbb{E}\left[\left(X_{T}-X_{T}^{v}\right)^{\top}q_{T}\right]+\mathbb{E}\left[\left(X_{T}-X_{T}^{v}\right)^{\top}M_{T}^{\top}p_{T}\right] \notag \\
	&=\mathbb{E}\int_{0}^{T}\Big\{\left(X_{T}-X_{T}^{v}\right)^{\top}\Big(-f_{x}^{\top}\left(t,E_{t},X_{t},u(t)\right)q_{t} \notag \\
  &\quad+\int_{|z|<c}g_{a}^{\top}\left(t,E_{t},X_{t},u(t),z\right)p_{t}\Pi(dz)-\sigma_{x}^{\top}\left(t,E_{t},X_{t},u(t)\right)k_{t}\notag \\
  &\quad-\int_{|z|<c}b_{x}^{\top}\left(t,E_{t},X_{t},u(t),z\right)R\left(t,z\right)\Pi(dz)-l_{x}^{\top}\left(t,E_{t},X_{t},u(t)\right)\Pi(dz) \notag \\
  &\quad+\langle q_{t},f\left(t,E_{t},X_{t},u(t)\right)-f\left(t,E_{t},X_{t}^{v},v(t)\right)\rangle+\langle k_{t},\sigma\left(t,E_{t},X_{t},u(t)\right)-\sigma\left(t,E_{t},X_{t}^{v},v(t)\right)\rangle \notag \\
  &\quad+\int_{|z|<c}R\left(E_{t},z\right)\left(b\left(t,E_{t},X_{t},u(t),z\right)-b\left(t,E_{t},X_{t}^{v},v(t),z\right)\right)\Pi(dz)\Big)\Big\} dE_{t} \notag \\
  &\quad+\mathbb{E}\left[\left(X_{T}-X_{T}^{v}\right)^{\top}M_{T}^{\top}p_{T}\right]. \notag 
\end{align}
And similarly, by \autoref{assumption2} and Itô's formula, we obtain
\begin{align}
  &\quad \;\mathbb{E}\left[\gamma\left(y_{0}\right)-\gamma\left(y_{0}^{v}\right)\right]  \notag \\
  &\leq\mathbb{E}\left[\left(\gamma\left(y_{0}\right)-\gamma\left(y_{0}^{v}\right)\right)^{\top}\gamma_{y}\left(y_{0}\right)\right] \notag \\
  &=-\mathbb{E}\left[\left(Y_{0}-Y_{0}^{v}\right)^{\top}p_{0}\right] \notag \\
  &=-\mathbb{E}\left[\left(X_{T}-X_{T}^{v}\right)^{\top}M_{T}^{\top}p_{T}\right] \notag \\
  &\quad+\mathbb{E}\int_{0}^{T}\int_{|z|<c}\Big[\left(Y_{t}-Y_{t}^{v}\right)^{\top}\left(g_{y}^{\top}\left(t,E_{t},X_{t},Y_{t},A_{t},r(t,\cdot),u(t),z\right)p_{t}-l_{y}^{\top}\left(t,E_{t},X_{t},Y_{t},A_{t},r(t,\cdot),u(t)\right)\right) \notag \\
  &\quad+\left(A_{t}-A_{t}^{v}\right)^{\top}\left(g_{a}^{\top}\left(t,E_{t},X_{t},Y_{t},A_{t},r(t,\cdot),u(t),z\right)p_{t}-l_{a}^{\top}\left(t,E_{t},X_{t},Y_{t},A_{t},r(t,\cdot),u(t)\right)\right) \notag \\
  &\quad+\left(r\left(t,z\right)-r^{v}\left(t,z\right)\right)^{\top}\left(g_{r}^{\top}\left(t,E_{t},X_{t},Y_{t},A_{t},r(t,\cdot),u(t),z\right)p_{t}-l_{r}^{\top}\left(t,E_{t},X_{t},Y_{t},A_{t},r(t,\cdot),u(t)\right)\right) \notag \\
  &\quad-\langle p_{t},g\left(t,E_{t},X_{t},Y_{t},A_{t},r(t,\cdot),u(t),z\right)-g\left(t,E_{t},X_{t}^{v},Y_{t}^{v},A_{t}^{v},r^{v}(t,\cdot),v(t)\right)\rangle\Big]\Pi(dz)dE_{t}. \notag 
\end{align}
Adding $I_1$ and $I_2$, we have
\begin{align}
  J(u(\cdot))-J(v(\cdot)) & \leq\mathbb{E}\int_{0}^{T}\Big[H\left(t,E_{t},X_{t},Y_{t},A_{t},r(t,\cdot),u(t),p_{t},k_{t},q_{t},R(t,\cdot)\right) \notag \\
  &\quad -H\left(t,E_{t},X_{t}^{v},Y_{t}^{v},A_{t}^{v},r^{v}(t,\cdot),v(t),p_{t},k_{t},q_{t},R(t,\cdot)\right)\notag \\
 & \quad-\langle H_{x}\left(t,E_{t},X_{t},Y_{t},A_{t},r(t,\cdot),u(t),p_{t},k_{t},q_{t},R(t,\cdot)\right),X_{t}-X_{t}^{v}\rangle\notag \\
 & \quad-\langle H_{y}\left(t,E_{t},X_{t},Y_{t},A_{t},r(t,\cdot),u(t),p_{t},k_{t},q_{t},R(t,\cdot)\right),Y_{t}-Y_{t}^{v}\rangle\notag \\
 & \quad-\langle H_{a}\left(t,E_{t},X_{t},Y_{t},A_{t},r(t,\cdot),u(t),p_{t},k_{t},q_{t},R(t,\cdot)\right),A_{t}-A_{t}^{v}\rangle\notag \\
 & \quad-\left\langle H_{r}\left(t,E_{t},X_{t},Y_{t},A_{t},r(t,\cdot),u(t),p_{t},k_{t},q_{t},R(t,\cdot)\right),r(t,\cdot)-r^{v}(t,\cdot)\right\rangle \Big]dE_{t}. \label{4.12}
\end{align}
Since by assumption $H$ is convex in $(x,y,a,r(\cdot),v)$, then the first two terms of above turns into
\begin{align}
  &\quad \; H\left(t,E_{t},X_{t},Y_{t},A_{t},r(t,\cdot),u(t),p_{t},k_{t},q_{t},R(t,\cdot)\right)-H\left(t,E_{t},X_{t}^{v},Y_{t}^{v},A_{t}^{v},r^{v}(t,\cdot),v(t),p_{t},k_{t},q_{t},R(t,\cdot)\right) \notag \\
 & \leq\langle H_{x}\left(t,E_{t},X_{t},Y_{t},A_{t},r(t,\cdot),u(t),p_{t},k_{t},q_{t},R(t,\cdot)\right),X_{t}-X_{t}^{v}\rangle\notag \\
 & \quad+\langle H_{y}\left(t,E_{t},X_{t},Y_{t},A_{t},r(t,\cdot),u(t),p_{t},k_{t},q_{t},R(t,\cdot)\right),Y_{t}-Y_{t}^{v}\rangle\notag \\
 & \quad+\langle H_{a}\left(t,E_{t},X_{t},Y_{t},A_{t},r(t,\cdot),u(t),p_{t},k_{t},q_{t},R(t,\cdot)\right),A_{t}-A_{t}^{v}\rangle\notag \\
 & \quad+\langle H_{r}\left(t,E_{t},X_{t},Y_{t},A_{t},r(t,\cdot),u(t),p_{t},k_{t},q_{t},R(t,\cdot)\right),r(t,\cdot)-r^{v}(t,\cdot)\rangle\notag \\
 & \quad+\langle H_{v}\left(t,E_{t},X_{t},Y_{t},A_{t},r(t,\cdot),u(t),p_{t},k_{t},q_{t},R(t,\cdot)\right),u(t)-v(t)\rangle.  \label{4.13}
\end{align}
Subsituting \eqref{4.13} into \eqref{4.12} gives
\begin{equation}
  J(u(\cdot))-J(v(\cdot))\leq\mathbb{E}\int_{0}^{T}\langle H_{v}\left(t,E_{t},X_{t},Y_{t},A_{t},r(t,\cdot),u(t),p_{t},k_{t},q_{t},R(t,\cdot)\right),u(t)-v(t)\rangle dE_{t}.
\end{equation}
Then, from \eqref{maximum condition}, we have $J(u(\cdot))\leq J(v(\cdot))$ for all $v(\cdot)\in\mathcal{U}$. Thus $u(\cdot)$ is optimal.
\end{proof}

\section{Application to a cash management problem}
We study a type of cash management problem for application.
Consider a company operating on a stochastic business time scale. The dynamics of its cash flow and the utility derived by management from a control strategy are described by the following system of partially coupled, forward-backward stochastic differential equations (FBSDEs) with Lévy noise:

\begin{align}
  \begin{cases}
dX_{t}^{v}=\left(-\mu_{1}X_{t}^{v}+\beta_{1}v\left(t\right)\right)dE_{t}+\sigma_{t}v\left(t\right)dB_{E_{t}}\\
\quad\quad+\int_{|z|<c}\eta_{t}(z)v\left(t-\right)\tilde{N}\left(dz,dE_{t}\right),\\
-dY_{t}^{v}=\int_{|z|<c}\left(-\mu_{1}Y_{t}^{v}+\mu_{2}X_{t}^{v}+\beta_{2}v\left(t\right)\right)\Pi\left(dz\right)dE_{t}-A_{t}^{v}dB_{E_{t}}\\
\quad\quad-\int_{|z|<c}r^{v}\left(t,z\right)\tilde{N}\left(dz,dE_{t}\right),\\
X_{0}^{v}=x_{0},\\
Y_{T}^{v}=X_{T}^{v}. 
\end{cases}  \label{ex: cash management}
\end{align}

where
constants $x_{0}\in\mathbb{R},\mu_{1},\mu_{2},\beta_{1},\beta_{2},\sigma_{t}>0$,
$X_t^v$ is the cash flow of an agent,
$v(t)$ is a control strategy of the agent and is regarded as the rate of capital injection or withdrawal,
$Y_t^v$ is the utility from $v(\cdot)$,
$(A_t^v)^2$ is the volatility of utility.

For any $v(\cdot)\in\mathcal{U}_{ad}$, \eqref{ex: cash management} has a unique solution $\left(X_{\cdot}^v,Y^v_{\cdot},A^v_{\cdot},r^v_(\cdot,\cdot)\right)$.

Introduce a cost functional:
\begin{equation}
  J(v(\cdot)):=\mathbb{E}\left[\frac{1}{2}\int_{0}^{T}(v(t)-\kappa(t))^2dE_{t}-y_{0}^{v}\right].
\end{equation}
where $\kappa(t)$ is a deterministic and bounded function with value in $\mathbb{R}$, serving as a dynamic benchmark.
Then, the cash mangement problem with stochastic recursive utility is as follows.

\begin{Problem}\label{problem: cash management}
   Find an optimal control strategy $u(\cdot)\in\mathcal{U}_{ad}$ such that
  \begin{equation}
    J(u(\cdot))=\inf_{v(\cdot)\in\mathcal{U}_{ad}}J(v(\cdot)).
  \end{equation}
  subject to \eqref{ex: cash management}.
\end{Problem}

We can check that \autoref{assumption} is satisfied. 
Then we can use the maximum principle \autoref{thm: Necessary condition} to solve the above problem.

The Hamiltonian function and the adjoint equation are then reduced to
\begin{align}
  H\left(t_{1},t_{2},x,y,a,r(\cdot),v,p,q,k,R(\cdot)\right)
  &=\left\langle q,-\mu_{1}x+\beta_{1}v\right\rangle +\left\langle k,\sigma v\right\rangle  \notag \\
  &\quad -\left\langle p,\left(-\mu_{1}y+\mu_{2}x+\beta_{2}v\right)\right\rangle 	\notag \\
	&\quad+\int_{|z|<c}\left\langle R(z),\eta_{t}(z)v\right\rangle \Pi(dz)dE_{t} \notag \\
	&\quad-\frac{1}{2}\int_{0}^{T}(v(t)-\kappa(t))^2. \notag 
\end{align}

and 
\begin{align}
  \begin{cases}
dp_{t}=-\mu_{1}p_{t}dE_{t},\\
-dq_{t}=\left(-\mu_{1}q_{t}-\mu_{2}p_{t}\right)dE_{t}-k_{t}dB_{E_{t}}-\int_{|z|<c}R_{t}\left(t,z\right)\tilde{N}(dz,dE_{t}),\\
p_{0}=1,\\
q_{T}=-p_{T}.
\end{cases} \label{ex: cash management-adjoint equation}
\end{align}

By \autoref{assumption} and \autoref{assumption2}, \eqref{ex: cash management-adjoint equation} admits a unique solution $\left(p_{\cdot},q_{\cdot},k_{\cdot},R\left(\cdot,\cdot\right)\right)$.

By the maximum condition \eqref{maximum condition}, we have
\begin{align}
  H_{v}\left(t_{1},t_{2},x,y,a,r(\cdot),v,p,q,k,R(\cdot)\right)
  &=q\beta_{1}+k\sigma_{t}-p\beta_{2}+\int_{|z|<c}R(z)\eta_{t}(z)\Pi(dz) \notag \\
  &\quad -(v(t)-\kappa(t)). \notag
\end{align}

Letting $H_v=0$, we have the expression of optimal control strategy:
\begin{equation}
  u(t)=\kappa(t)-\left(\beta_{1}q_{t}+\sigma_{t}k_{t}-\beta_{2}p_{t}+\int_{|z|<c}R(t,z)\eta_{t}(z)\Pi(dz)\right). \label{optimal control}
\end{equation}

\begin{Proposition}
  The optimal strategy of Problem \autoref{problem: cash management} is given by \eqref{optimal control}, where $\left(p_{\cdot},q_{\cdot},k_{\cdot},R\left(\cdot,\cdot\right)\right)$ solves \eqref{ex: cash management-adjoint equation}.
\end{Proposition}

\section*{Acknowledgements}  
The author would like to thank the editors and referees for their professional suggestions. This research did not receive any specific grant
from funding agencies in the public, commercial, or not-for-profit sectors.

\section*{Data Availability}  
No data was used for the research described in the article.

\bibliographystyle{elsarticle-num}  
\bibliography{main}

\end{document}